%% file: manuskript.tex
\documentclass[a4paper,10pt]{amsart}

\usepackage[utf8]{inputenc}
\usepackage[T1]{fontenc}
\usepackage[english]{babel}

\usepackage[margin=7em]{geometry}

\usepackage{amssymb, amsfonts}

\usepackage{enumerate}

\usepackage{float}
\usepackage{booktabs}

\usepackage{listings}

\usepackage{url}

\numberwithin{equation}{section}

\newtheorem{theoremcounter}{theoremcounter}[section]

\newtheorem{conjecture}[theoremcounter]{Conjecture}
\newtheorem{corollary}[theoremcounter]{Corollary}
\newtheorem{definition}[theoremcounter]{Definition}
\newtheorem{example}[theoremcounter]{Example}
\newtheorem{proposition}[theoremcounter]{Proposition}
\newtheorem{theorem}[theoremcounter]{Theorem}

\input{shortcuts}

\begin{document}

\title{$M_{24}$-twisted Product Expansions are Siegel Modular Forms}

\author{Martin Raum}
\address{Max Planck Institut f\"ur Mathematik\\Vivatsgasse 7\\53111 Bonn, Germany}
\email{MRaum@mpim-bonn.mpg.de}

\curraddr{ETH, Dept. Mathematics, Rämistraße 101, CH-8092, Zürich, Switzerland}
\thanks{While this paper was written the author was holding a scholarship from the Max Planck society.  The author is currently supported by the ETH Zurich Postdoctoral Fellowship Program and by the Marie Curie Actions for People COFUND Program.}



\begin{abstract}
Cheng constructed product expansions from twists of elliptic genera of symmetric powers of $K3$ surfaces that are related to $M_{24}$ moonshine.  We study which of them are Siegel modular forms.  If the predicted level is non-composite, they are modular, and their powers can be represented as products of rescaled Borcherds products.
\end{abstract}

\maketitle

\section{Introduction}
\label{sec:introduction}

In~\cite{CN79}, Conway and Norton described a phenomenon that later became famous as ``Monstrous Moonshine'' and was completely resolved by Borcherds~\cite{Bo92} after parts had been proven in~\cite{FLM88}.  He showed that (twisted) McKay-Thompson series of the ``Monster Algebra'' are modular forms.  Monstrous moonshine connects the realm of operator algebras, more precisely vertex operator algebras~\cite{LL04}, with the world of modular forms~\cite{BGHZ08}.  This connection was stimulating to both.  In~\cite{EOT10}, Eguchi, Ooguri, and Tachikawa discovered another moonshine phenomenon originating in elliptic genera of $K3$, which they linked to the Mathieu group~$M_{24}$.  They conjectured that it gives rise to mock modular forms.  In this paper, we address a question on modularity that arises in this context. 

Mock modular forms generalize holomorphic elliptic modular forms~\cite{Zw02, BF04, Za06, On09}, and have been successfully applied, e.g., in physics and combinatorics~\cite{BO06, Ma10}.  The findings of Eguchi, Ooguri, and Tachikawa relate the first few Fourier coefficients of a mock modular form arising naturally from the elliptic genus of $K3$ to sums of dimensions of representations of~$M_{24}$.  Motivated by this observation, they conjectured that all Fourier coefficients are decomposable in such a way.  This conjecture attracted much interest in physics and mathematics; See, for example,~\cite{GV12}.  In the mean time, the initial observation made in~\cite{EOT10} has been confirmed by Gannon~\cite{Ga12}.

It is commonly believed that the above modularity results come from a vertex operators algebra which carries an $M_{24}$ action.  While the questions on decompositions of Fourier coefficients of mock modular forms posed in~\cite{EOT10} were answered in the affirmative, the potential connection to vertex operators algebras, and thus string theory, remains to be examined.

Cheng suggested to construct $M_{24}$ twisted elliptic genera for symmetric powers of $K3$ surfaces~\cite{Che10}, which must be modular if a vertex operator algebra as above exists.  She gave explicit product expansions $\Phi_g$ attached to conjugacy classes $g$ of $M_{24}$, and conjectured they are Siegel modular forms of degree~$2$.  They are, she argued, related to the $1/4$-BPS spectrum of the $\text{K}3 \times T^2$-compactified type II string theory~\cite{Che10}.  In contrast to the expansions given in~\cite{EH12}, they are, however, not modular by construction.

We briefly set up notation to state Conjecture~\ref{conj:cheng-duncan} and Theorem~\ref{maintheorem}.  Conjugacy classes (or pairs of conjugacy classes) of $M_{24}$ that appear are labeled according to ATLAS~\cite{Co85} by
\begin{gather}
\label{eq:M24-conjugacy-classes}
\begin{split}
&
  1\text{A}, 2\text{A}, 2\text{B}, 3\text{A}, 3\text{B},
  4\text{A}, 4\text{B}, 4\text{C}, 5\text{A}, 6\text{A}, 6\text{B},
\\
&
  7\text{AB}, 8\text{A}, 10\text{A}, 11\text{A}, 12\text{A}, 12\text{B},
  14\text{AB}, 15\text{AB}, 21\text{AB}, 23\text{AB}
\text{.}
\end{split}
\end{gather}
Siegel modular forms generalize the notion of elliptic modular forms, and appear, for example, in applications to moduli problems~\cite{Ig67} and state counting in the theory of quantum black holes~\cite{DMZ12}.  We define Siegel modular forms of degree~$2$:  Let
\begin{gather*}
  \HS_2
=
   \{ Z \in \Mat{2}{\CC} \,:\, Z^\T = Z,\, \Im(Z) \text{ positive definite}\}
\end{gather*}
be the Siegel upper half space of degree~$2$. A meromorphic Siegel modular form (of degree~$2$) of weight~$k \in \ZZ$ for a finite index subgroup $\Gamma \subseteq \Sp{2}(\ZZ)$ and a character $\chi$ of $\Gamma$ is a meromorphic function $\Phi :\, \HS_2 \rightarrow \CC$ that satisfies
\begin{gather*}
  \Phi\big( (A Z + B) (C Z + D)^{-1} \big)
=
  \det(C Z + D)^k
  \chi(\left(\begin{smallmatrix} A & B \\ C & D \end{smallmatrix}\right))\,
  \Phi(Z)
\end{gather*}
for all $\left(\begin{smallmatrix} A & B \\ C & D \end{smallmatrix}\right) \in \Gamma$.  The subgroups $\Gamma_0^{(2)}(N)$ ($0 < N \in \ZZ$) of all matrices in $\Sp{2}(\ZZ)$ with $C \in N \cdot \Mat{2}{\ZZ}$ play a major role in the present work.  We say a Siegel modular form has level~$N$ if it is a Siegel modular form for $\Gamma^{(2)}_0(N)$ and some character.

In~\cite{CD12}, Cheng and Duncan discussed more thoroughly the twisted elliptic genera that appeared in~\cite{Che10}.  They made the following precise conjecture.
\begin{conjecture}[{Cheng, Duncan~\cite{CD12}}]
\label{conj:cheng-duncan}
For all conjugacy classes $g$ that are given in~\eqref{eq:M24-conjugacy-classes}, the product~$\Phi_g$ (defined in~\eqref{eq:chengducan_productexpansion}) is a Siegel modular form of level~$N_g$, where $N_g$ is given in Table~\ref{tab:Ng}.
\end{conjecture}
The affirmative answer to this conjecture would mean significant support for the idea that a vertex operator algebra with $M_{24}$ action exists.  It also has an interpretation in terms of root multiplicities of generalized Kac-Moody algebras~\cite{Che10}.

The cases $g \in \{ 1\text{A}, 2\text{A}, 3\text{A}, 4\text{B} \}$ of Conjecture~\ref{conj:cheng-duncan} were solved using Borcherds products (see, e.g.,~\cite{CG08}).  We are able to resolve all cases for which $N_g$ is a prime power.  More precisely, we prove the following result:
\begin{theorem}
\label{maintheorem}
For
\begin{gather*}
  g
\in
  \big\{ 1\text{A}, 2\text{A}, 2\text{B}, 3\text{A}, 3\text{B},
         4\text{A}, 4\text{B}, 4\text{C}, 5\text{A},
         7\text{AB}, 8\text{A}, 11\text{A}, 23\text{AB}
  \big\}
\text{,}
\end{gather*}
Conjecture~\ref{conj:cheng-duncan} is true.
\end{theorem}
\begin{table}[H]
\framebox[13em]{
\begin{tabular}{l@{\hspace{2em}}ll}
\ \\
\toprule
$g$   & $N_g$ & $k_g$ \\
\midrule
$1\text{A}$   & $1$ & $10$ \\
$2\text{A}$   & $2$ & $6$ \\
$2\text{B}$   & $4$ & $4$ \\
$3\text{A}$   & $3$ & $4$ \\
$3\text{B}$   & $9$ & $2$ \\
$4\text{A}$   & $8$ & $2$ \\
$4\text{B}$   & $4$ & $3$ \\
$4\text{C}$   & $16$ & $1$ \\
$5\text{A}$   & $5$ & $2$ \\
$7\text{AB}$  & $7$ & $1$ \\
$8\text{A}$   & $8$ & $\frac{1}{2}$ \\
$11\text{A}$  & $11$ & $0$ \\
$23\text{AB}$ & $23$ & $-1$ \\
\bottomrule
\\
\end{tabular}
}\vspace{0.5em}
\caption{The level~$N_g$ and the weight~$k_g$ of $\Phi_g$.  A proof is given in Theorem~\ref{thm:phig_modularity} and Corollary~\ref{cor:phig_modularity}.}
\label{tab:Ng}
\end{table}
We remark that the weights~$k_g$ of $\Phi_g$, given in Table~\ref{tab:Ng}, are not to be confused with the weights~$k_g$ that are given in Table~1 of~\cite{CD12}.  The latter are simply half the number of cycles in the conjugacy class~$g$.  Already in the case~$g = \text{1A}$ the two quantities, which then equal $10$ and $12$, are not the same.

The functions $\Phi_g$ given by Cheng and Duncan are product expansions.  In the case $g \in \{ 1\text{A}, 2\text{A}, 3\text{A}, 4\text{B} \}$ they were known to be Borcherds products.  Borcherds products, studied in~\cite{Bo98}, are the only basic construction of modular product expansions that are currently known.  Clearly, products of Borcherds products are modular, too, and so are rescaled Borcherds functions $\Phi(N Z)$ ($0 < N \in \ZZ$), where $\Phi(Z)$ is a Borcherds product.  Thus it is natural to believe that all $\Phi_g$, if they are modular, are products of rescaled Borcherds products.  If they are not, but Conjecture~\ref{conj:cheng-duncan} is still true, it would provide a novel basic construction of modular product expansions, which is unlikely.  For this reason, we have studied the question whether or not the $\Phi_g$ are products of rescaled Borcherds products.

Our proof of Theorem~\ref{maintheorem} relies on linearization of product expansions.  We associate to each product expansions $\Phi$ a vector $\cE(\Phi)$ in a $\ZZ$-module~$E$ that is defined in Section~\ref{sec:preliminaries}.  For any $0 < N \in \ZZ$, we can identify explicitly a submodule~$E_{\rm Bor}(N) \subset E$ associated with products of rescaled Borcherds products of level~$N$.  Proving Theorem~\ref{maintheorem} then amounts to checking whether the rank one module $\ZZ \cdot \cE(\Phi_g)$ has nontrivial intersection with~$E_{\rm Bor}(N_g)$.  However, it is a difficult task to compute $\cE_{\rm Bor}(N)$, which, if $N$ is not squarefree, involves computation of Fourier expansions of elliptic modular forms of non-squrefree level at all cusps.  We employ Sage~\cite{sage} and~\cite{Ra13} to do these computations.
\vspace{1ex}

The paper is organized as follows.  Section~\ref{sec:preliminaries} contains preliminaries on modular forms and product expansions.  In Section~\ref{sec:extendedborcherdsproducts}, we recall the theory of Borcherds products that we will make use of.  We also introduce rescaled Borcherds products.  Section~\ref{sec:chengducan} contains a revision of the material in~\cite{CD12} that is relevant to this paper.  The proof of Theorem~\ref{thm:phig_modularity} is discussed in Section~\ref{sec:results}.  Final remarks are given in Section~\ref{sec:conclusion}.

We have included three appendices.  Appendix~\ref{appendix:how-we-found-the-data} contains a description of how we found, by computer methods, the solutions to the Cheng's and Duncan's modularity problem.  In Appendix~\ref{appendix:projection-matrices}, we briefly describe numerical tests performed to verify data in Table~\ref{tab:projected-cusp-expansions}.  Most tables included in this paper are given in Appendix~\ref{appendix:tables}.

{\tit Acknowledgement:} The author thanks Kathrin Bringmann for helpful conversations about the product expansions in~\cite{CD12}.  He is grateful to \"Ozlem Imamo\u glu for comments on an earlier version of this paper.

\section{Preliminaries}
\label{sec:preliminaries}

Denote the space of elliptic modular forms of weight~$k$ for a finite index subgroup $\Gamma \subset \SL{2}(\ZZ)$ by $\rmM_k(\Gamma)$.  We write $\rmM^{!}_k(\Gamma)$ for the space of weakly holomorphic modular forms of weight~$k$.  Let $\HS \subset \CC$ be the Poincar\'e upper half plane.  The weight~$k$ slash action is denoted by $|_k$.  Precise definitions and a useful introduction into the subject can be found in~\cite{BGHZ08}.

We write $\rmJ_{k, m}(\Gamma) \subseteq \rmJ^{(!)}_{k, m}(\Gamma) \subseteq \rmJ^{!}_{k, m}(\Gamma)$ for the space of Jacobi forms, weak Jacobi forms, and weakly holomorphic Jacobi forms of weight~$k$ and index~$m$ for $\Gamma \ltimes \ZZ^2$.  Jacobi forms are functions on the Jacobi upper half plane $\HS^\rmJ := \HS \times \CC$.  The reader is referred to~\cite{EZ85} for definitions.

Fourier expansions of elliptic modular forms and weak Jacobi forms are central in this paper.  Write $q := e(\tau)$, $\zeta := e(z)$, $e(x) := \exp(2 \pi i \, x)$, where $\tau$ and $(\tau, z)$ are coordinates of the Poincar\'e and the Jacobi upper half plane, respectively.  Write $\pi_{\rm FE}$ for the projection of the Fourier expansion to those terms with integral exponents:
\begin{gather*}
  \pi_{\rmF\rmE} \Big( \sum_{n \in \QQ} c(n) \, q^n \Big)
:=
  \sum_{n \in \ZZ} c(n) \, q^n
\quad
\text{and}
\quad
  \pi_{\rmF\rmE} \Big( \sum_{n \in \QQ,\, r \in \ZZ} c(n, r) \, q^n \zeta^r \Big)
:=
  \sum_{n \in \ZZ\, r \in \ZZ} c(n, r) \, q^n \zeta^r
\text{.}
\end{gather*}

For $0 < N \in \ZZ$, denote the set of cusps of $\Gamma_0(N) \mathop{\backslash} \HS$ by $\cC(N)$.  Fix a cusp $\frakc \in \cC(N)$ and a matrix $\gamma \in \SL{2}(\ZZ)$ such that $\gamma \infty = \frakc$.  For a modular form~$f$ of weight~$k$ or a weak Jacobi form~$\phi$ of weight~$k$ and index~$m$ we write
\begin{gather*}
  \pi_{\rm FE} (f_\frakc)
:=
  \pi_{\rm FE} \big( f \big|_k \gamma \big)
\quad
\text{and}
\quad
  \pi_{\rm FE} (\phi_\frakc)
:=
  \pi_{\rm FE} \big( \phi \big|_{k, m} \gamma \big)
\text{.}
\end{gather*}
This notation is well-defined, since the left hand sides only depend on $\frakc$, but not on $\gamma$.

Given a cusp $\frakc \in \cC(N)$, we write $h_\frakc(N)$ and $e_\frakc(N)$ for the width and the denominator of $\frakc$, respectively.  Fix some $0 \le f_\frakc(N) \le e_\frake(N)$ such that $\frakc = f_\frakc \slashdiv e_\frakc$ as cusps of $\Gamma_0(N) \mathop{\backslash} \HS$.  Set $N_\frakc(N) := N e_\frakc(N)^{-1}$.

\subsection{Elliptic modular forms}
\label{ssec:preliminaries:elliptic-modular-forms}

As a shorthand notation, we will write $\rmM_k(N)$ for $\rmM_k(\Gamma_0(N))$.  Fix $k$ and $N$, and let $d = \dim\, \rmM_k(N)$.  Recall that the echelon basis of elliptic modular forms of weight~$k$ for $\Gamma_0(N)$ consists of elements $f_{k, N; 1}, \ldots, f_{k, N; d} \in \rmM_k(N)$ with Fourier expansions
\begin{gather*}
  f_{k, N; d}(\tau)
=
  \sum_n c(f_{k, N; d}; n)\, q^n
\end{gather*}
such the matrix $\big(c(f_{k, N}^{(d)}; n)\big)_{1 \le i \le d;\, 0 \le n}$ has echelon form.  Define
\begin{gather*}
  \rmM_k(N)[ c_1, \ldots, c_d ]
=
  \sum_{i = 1}^d c_i\, f_{k, N; i}
\text{.}
\end{gather*}
We use this notation frequently in Section~\ref{sec:results}, and in particular, in Table~\ref{tab:borcherds_products}.  We will also identify elliptic modular forms with coordinate (column) vectors with respect to this basis.
\begin{example}
The echelon basis of $\rmM_2(11)$ has form
\begin{align*}
  f_{2, 11; 1}(\tau)
&=
  1 + 12 q^2 + 12 q^3 + 12 q^4 + 12 q^5 + O(q^6)
\text{,}
\\
  f_{2, 11; 2}(\tau)
&=
  q - 2 q^2 - q^3 + 2 q^4 + q^5 + O(q^6)
\text{.}
\end{align*}
Correspondingly, we have
\begin{gather*}
  \rmM_2(11)[c_1, c_2]
=
  c_1 + c_2 q + (12 c_1 - 2 c_2) q^2 + (12 c_1 - c_2) q^3 +  (12 c_1 + 2 c_2) q^4 + (12 c_1 + c_2) q^5 + O(q^6)
\text{.}
\end{gather*}
\end{example}

\begin{proposition}
Suppose that $f \in \rmM_k(N)$ for some $0 < N \in \ZZ$.  If $\frakc \in \cC(N)$ has a representative of the form $1 \slashdiv e \in \QQ$, then $\pi_{\rm FE}(f_\frakc) \in \rmM_k(\Gamma)$, where
\begin{gather*}
  \Gamma
=
  \big\{ \left(\begin{smallmatrix} a & b \\ c & d \end{smallmatrix}\right) \in \Gamma_0(N) \,:\,
     a - d \equiv 0 \pmod{\gcd(e, e^{-1}N)} \big\}
\text{.}
\end{gather*}

In particular, if the odd part of $N$ is squarefree and the even part divides $8$, then $\pi_{\rm FE}(f_\frakc) \in \rmM_k(\Gamma_0(N))$ for all $\frakc \in \cC(N)$.
\end{proposition}
\begin{proof}
Let $\frakc \in \cC(N)$ with representative $1 \slashdiv e$, as in the statement.  By definition,
\begin{gather*}
  \pi_{\rm FE}( f_\frakc )
=
  \pi_{\rm FE}( f \big|_k \left(\begin{smallmatrix} 1 & 0 \\ e & 1 \end{smallmatrix}\right) )
=
  \sum_{n \pmod{h_\frakc(N)}}
  f \big|_k \left(\begin{smallmatrix} 1 & 0 \\ e & 1 \end{smallmatrix}\right)
            \left(\begin{smallmatrix} 1 & n \\ 0 & 1 \end{smallmatrix}\right)
\text{.}
\end{gather*}

For $\left(\begin{smallmatrix} a & b \\ c & d \end{smallmatrix}\right) \in \Gamma$ and any choices of $n'(n)$, we have
\begin{gather*}
  \sum_{n \pmod{h_\frakc(N)}}
  \Big(
  f \big|_k \left(\begin{smallmatrix} 1 & 0 \\ e & 1 \end{smallmatrix}\right)
            \left(\begin{smallmatrix} 1 & n \\ 0 & 1 \end{smallmatrix}\right)
            \left(\begin{smallmatrix} a & b \\ c & d \end{smallmatrix}\right)
            \left(\begin{smallmatrix} 1 & -n'(n) \\ 0 & 1 \end{smallmatrix}\right)
            \left(\begin{smallmatrix} 1 & 0 \\ -e & 1 \end{smallmatrix}\right)
  \Big) \Big|_k
            \left(\begin{smallmatrix} 1 & 0 \\ e & 1 \end{smallmatrix}\right)
            \left(\begin{smallmatrix} 1 & n'(n) \\ 0 & 1 \end{smallmatrix}\right)
\text{.}
\end{gather*}
We shall want to choose $n'(n)$ such that it runs through a system of representatives ${\rm mod}\; h_\frakc(N)$ as $n$ does, and the bottom left entry of the inner matrix product is divisible by $N$.  Then we can sum over $n'$ instead of $n$, and we will have shown that
\begin{gather*}
  \pi_{\rm FE}( f_\frakc ) \big|_k 
  \left(\begin{smallmatrix} a & b \\ c & d \end{smallmatrix}\right)
=
  \pi_{\rm FE}( f_\frakc )
\text{,}
\end{gather*}
which proves the desired statement.

The divisibility condition on the bottom left entry of the inner matrix product is equivalent to
\begin{gather*}
  e a \,n'(n)
\equiv
  a - d - e b - e d n
  \pmod{e^{-1}N}
\text{.}
\end{gather*}
By the condition on $a - d$, this is the same as
\begin{gather*}
  \frac{e}{\gcd(e, e^{-1}N)} a n'(n)
\equiv
  \frac{a - d - e b - e d n}{\gcd(e, e^{-1}N)}
  \pmod{\gcd(e^2, N)^{-1} N}
\text{.}
\end{gather*}
We have $h_\frakc(N) = \gcd(e^2, N)^{-1} N$, so the congruence condition is a condition ${\rm mod}\; h_\frakc(N)$.  Using the Chinese Remainder Theorem, we can and will assume that $N$ is a prime power.  It is straightforward to check that either $e = \gcd(e, e^{-1}N)$ or $\gcd(e, e^{-1}N)^{-1} f$ is a unit ${\rm mod}\; h_\frakc(N)$.  Since, further, $a$ is coprime to $N$, we can choose
\begin{gather*}
  n'(n)
\equiv
  a^{-1} \Big(
  \big(\frac{e}{\gcd(e, e^{-1}N)}\big)^{-1}
  \frac{a - d}{\gcd(e, e^{-1}N)}
  - b - d n
  \Big)
  \pmod{h_\frakc(N)}
\text{.}
\end{gather*}
Because $d$ is coprime to $N$, too, we find that $n'(n)$ runs through a system of representatives ${\rm mod}\;h_\frakc(N)$ as $n$ does.  This completes the proof.
\end{proof}

From the previous proposition, we find that
\begin{gather*}
  \pi_{\rm FE} (\,\cdot\,_\frakc) :\, \rmM_k(N) \rightarrow \rmM_k(N)
\end{gather*}
for $N \in \{2, 3, 4, 5, 7, 8, 11, 23\}$.  In Table~\ref{tab:projected-cusp-expansions}, we have given these maps as matrices $\Pi_{\rm FE}( k, N, \frakc )$ acting on coordinate (column) vectors with respect to the echelon basis.  They have been computed by the method described in~\cite{Ra13}, which uses algebraic methods.  In Appendix~\ref{appendix:projection-matrices}, we describe additional, numerical tests that have been performed to verify their correctness.

\begin{example}
The first column of $\Pi_{\rm FE}(2, 8; \frac{1}{2})$ equals $(-\frac{1}{8}\, 3\, -3)^\T$.  Therefore, we have
\begin{gather*}
  \pi_{rm FE}\big( \rmM_2(8)[1, 0, 0]_{\frac{1}{2}} \big)
=
  \rmM_2(8)\big[ -\tfrac{1}{8},\, 3,\, -3 \big]
\text{.}
\end{gather*}
More concretely, this means
\begin{gather*}
  \pi_{\rm FE} \big( \rmM_2(8)[1, 0, 0]_{\frac{1}{2}} \big)
=
  -\tfrac{1}{8} + 3q - 3q^{2} + 1332q^{3} + 2172q^{4} + 46026q^{5} + O(q^{6})
\end{gather*}
\end{example}

\subsection{Jacobi forms of index~$1$}

In this paper we only use weak Jacobi forms of weight~$0$ and index~$1$.  Such a Jacobi form~$\phi$ has a Fourier expansion
\begin{gather*}
  \sum_{0 \ge n,\, r \in \ZZ} c(\phi;\, 4 n - r^2) \, q^n \zeta^r
\text{,}
\end{gather*}
Recall the weak Jacobi forms~$\phi_{0, 1}$ and $\phi_{-2, 1}$ defined in~\cite{EZ85}, Theorem~9.3.  In the proof of Theorem~\ref{thm:phig_modularity}, we will need the initial Fourier expansions of $\phi_{-2, 1}$ and $\phi_{0, 1}$.  They equal
\begin{gather}
\label{eq:weak-phi-fourier-expansions}
  \phi_{-2, 1}(\tau, z)
=
  \zeta - 2 + \zeta^{-1} + O(q)
\text{,}
\quad
  \phi_{0, 1}(\tau, z)
=
  \zeta + 10 + \zeta^{-1} + O(q)
\text{.}
\end{gather}

Fix $0 < N \in \ZZ$.  Any $\phi \in \rmJ^{!}_{0, 1}(\Gamma_0(N))$ can be written as
\begin{gather*}
  \phi(\tau, z)
=
  tc_0(\phi)(\tau)\, \frac{\phi_{0,1}(\tau, z)}{12}
  + tc_2(\phi)(\tau)\, \phi_{-2, 1}(\tau, z)
\text{,}
\end{gather*}
where $tc_0(\phi) \in \rmM^{!}_0(\Gamma_0(N))$ and $tc_2(\phi) \in \rmM^{!}_2(\Gamma_0(N))$ are the (rescaled) $0$th and $2$nd Taylor coefficient of~$\phi$.  One can directly verify that $\phi$ is a weak Jacobi form in the sense of~\cite{EZ85} if $tc_0(\phi) \in \CC = \rmM_0(\Gamma_0(N))$ and $tc_2(\phi) \in \rmM_2(\Gamma_0(N))$.  This yields an isomorphism of vector spaces
\begin{gather}
\label{eq:jacobi-modular-forms-isomorphism}
  \rmJ^{(!)}_{0, 1}(\Gamma_0(N))
\rightarrow
  \rmM_0(\Gamma_0(N)) \times \rmM_2(\Gamma_0(N))
\text{,}\quad
  \phi
\mapsto
  \big(tc_0(\phi), tc_2(\phi)\big)
\text{.}
\end{gather}
We will use this in order to reduce our considerations to elliptic modular forms.

In order to apply Theorem~\ref{thm:borcherds-products}, which describes the product expansions that arise via Borcherds's construction, we have to compute Fourier expansions, $\pi_{\rm FE}(\phi_\frakc)$, at all cusps~$\frakc$ for weak Jacobi forms~$\phi$.  Since $\phi_{0, 1}$ and $\phi_{-2, 1}$ are weak Jacobi forms of level~$1$, we get
\begin{gather}
  \pi_{\rm FE}(\phi_\frakc)
=
  \pi_{\rm FE}\big( tc_0(\phi)_\frakc \big)\, \frac{\phi_{0,1}}{12}
  + \pi_{\rm FE}\big( tc_2(\phi)_\frakc \big)\, \phi_{-2, 1}
=
  tc_0(\phi)\, \frac{\phi_{0,1}}{12}
  + \phi_{\rm FE}\big( tc_2(\phi)_\frakc \big)\, \phi_{-2, 1}
\text{.}
\end{gather}

\subsection{Siegel product expansions}
\label{ssec:siegel-product-expansions}

Siegel modular forms (of genus~$2$) are certain functions on the Siegel upper half space
\begin{gather*}
  \HS_2
=
   \{ Z \in \Mat{2}{\CC} \,:\, Z^\T = Z,\, \Im(Z) \text{ positive definite}\}
\end{gather*}
We will write $Z = \left(\begin{smallmatrix} \tau_1 & z \\ z & \tau_2 \end{smallmatrix}\right)$ for the entries of $Z$, and $q_1 := e(\tau_1)$, $\zeta := e(z)$, and $q_2 := e(\tau_2)$ for the corresponding Fourier expansion variables.

Write $\left(\begin{smallmatrix} A & B \\ C & D\end{smallmatrix}\right)$, with $A, B, C, D \in \Mat{2}{\RR}$ for a typical element $\gamma \in \Sp{2}(\RR)$.
\begin{definition}
Let $f :\, \HS_2 \rightarrow \CC$ be a holomorphic function.  We call $f$ a Siegel modular form of weight~$k$ with character~$\chi$ for $\Gamma \subset \Sp{2}(\ZZ)$ if and only if
\begin{gather*}
  \Phi\big( (A Z + B) (C Z + D)^{-1} \big)
=
  \det(C Z + D)^k
  \chi(\left(\begin{smallmatrix} A & B \\ C & D \end{smallmatrix}\right))\,
  \Phi(Z)
\end{gather*}
for all $\gamma \in \Gamma$.
\end{definition}
\noindent
We refer the reader to~\cite{Fr83} for details on Siegel modular forms.  For $0 < N \in \ZZ$, we define
\begin{gather*}
  \Gamma^{(2)}_{0}(N)
:=
  \big\{ \left(\begin{smallmatrix} A & B \\ C & D \end{smallmatrix}\right)
         \,:\, C \equiv 0 \in \Mat{2}{\ZZ} \pmod{N} \big\}
\text{.}
\end{gather*}

Fix $0 < N_\Phi \in \ZZ$ and $e_{q_1}(\Phi), e_\zeta(\Phi), e_{q_2}(\Phi) \in \ZZ$.  Given weak Jacobi forms
\begin{gather}
  \psi[\Phi, d](\tau, z)
:=
  \sum_{n \ge 0,\, r \in \ZZ}
  c[\Phi, d](4 n - r^2) \, q^n \zeta^r
\in
  \rmJ^{(!)}_{0, 1}(\Gamma_0(N_\Phi))
\end{gather}
for $0 < d \in \ZZ$ that have integral Fourier coefficients and vanish for all but finitely many~$d$, we can define an absolutely convergent product expansion
\begin{gather*}
  \Phi(Z)
=
  q_1^{e_{q_1}(\Phi)} \zeta^{e_\zeta(\Phi)} q_2^{e_{q_2}(\Phi)}
  \prod_{d \isdiv n_\Phi}
  \prod_{(n, r, m) > 0}
  (1 - (q_1^n \zeta^r q_2^m)^d)^{c[\Phi, d](4 n m - r^2)}
\text{.}
\end{gather*}
All product expansions that show up in this paper are of this form.  Write $E$ for the space of functions $\ZZ_{> 0} \rightarrow \rmM_0(N_\Phi) \times \rmM_2(N_\Phi)$ with finite support.  We obtain $\cE(\Phi) \in E$:
\begin{align*}
  \cE(\Phi)
:\,
& {}
  \ZZ_{> 0} \rightarrow \rmM_0(N_\Phi) \times \rmM_2(N_\Phi)
\text{,}
\\[2pt]
  \cE(\Phi)(d)
={}&
  \begin{cases}
   \big( tc_0(\psi[\Phi, d]),\, tc_2(\psi[\Phi, d]) \big)
   \text{,}
   &
   \text{if $d \isdiv n_\Phi$;}
   \\
   0
   \text{,}
   &
   \text{otherwise.}
  \end{cases}
\end{align*}
Product expansions can be represented by $\cE(\Phi)$ and the triple $\big( e_{q_1}(\Phi), e_\zeta(\Phi), e_{q_2}(\Phi) \big)$.

Observe that
\begin{gather}
\label{eq:cE-additivity}
  \cE( \Phi \Phi' )
=
  \cE(\Phi) + \cE(\Phi')
\end{gather}
for product expansions $\Phi$ and $\Phi'$.  That is, the map $\cE$ linearizes our problem of identifying Cheng's and Duncan's product expansion as products of rescalded Borcherds products.

\section{Rescaled Borcherds products}
\label{sec:extendedborcherdsproducts}

We recall a special case of a theorem by Cl\'ery and Gritsenko~\cite{CG08}, which is a refinement and reformulation in terms of Jacobi forms, of Borcherds's result~\cite{Bo98} on product expansions for Siegel modular forms.

For a triple $(n, r, m)$ of integers, $(n, r, m) > 0$ means that $n > 0$, or $n = 0$ and $m > 0$, or $n = m = 0$ and $r < 0$.
\begin{theorem}[Cl\'ery-Gritsenko]
\label{thm:borcherds-products}
Let $\phi(\tau, z) = \sum_{n, r} c(4 n - r^2)\, q^n \zeta^r \in J_{0,1}^!(\Gamma_0(N))$.  Write $c(\phi_\frakc;\, \Delta)$ for the Fourier coefficients of $\phi$ at $\frakc$, and assume that for all $\frakc \in \cC(N)$ we have $h_\frakc(N) N_\frakc(N)^{-1} \, c(\phi_\frakc;\, \Delta) \in \ZZ$, if $\Delta \le 0$.  Let
\begin{gather*}
  B_N[\phi](Z)
=
  q_1^{e_{q_1}(\phi)} \zeta^{e_\zeta(\phi)} q_2^{e_{q_2}(\phi)}
  \prod_{\frakc \in \cC(N)}
  \prod_{(n, r, m) > 0} \big(1 - (q_1^n \zeta^r q_2^m)^{N_\frakc} \big)^{h_\frakc(N) N_\frakc(N)^{-1} \, c(\phi_\frakc;\, 4 n m - r^2)}
\end{gather*}
where
\begin{align*}
  e_{q_1}(\phi)
=
  e_{q_1}(B_N[\phi])
&=
  \frac{1}{24} \sum_{\substack{\frakc \in \cC(N) \\ l \in \ZZ}} \!\!
               h_\frakc(N) c(\phi_\frakc;\, -l^2)
\text{,}
\\[2pt]
  e_\zeta(\phi)
=
  e_\zeta(B_N[\phi])
&=
  \frac{1}{2} \sum_{\substack{\frakc \in \cC(N) \\ 0 < l \in \ZZ}} \!\!
              l \, h_\frakc(N) c(\phi_\frakc;\, -l^2)
\text{,}
\\[2pt]
  e_{q_2}(\phi)
=
  e_{q_2}(B_N[\phi])
&=
  \frac{1}{4} \sum_{\substack{\frakc \in \cC(N) \\ l \in \ZZ}} \!\!
              l^2 h_\frakc(N) c(\phi_\frakc;\, -l^2)
\text{.}
\end{align*}
Set
\begin{gather*}
  k
=
  \frac{1}{2} \sum_{\frakc \in \cC(N)} \!\!
              h_\frakc(N) N_\frakc(N)^{-1} \, c_\frakc(0)
\text{.}
\end{gather*}

Then $B_N[\phi]$ is a meromorphic Siegel modular form of weight~$k$ with character for $\Gamma^{(2)}_0(N)$.
\end{theorem}
\noindent
Given $0 < n \in \ZZ$, we write $B_N[\phi, n](Z) = B_N[\phi](n Z)$ for a \emph{rescaled Borcherds product}, which has level~$n N$.

The map $\cE(B_N[\phi, n]) \in E$ associated to $B_N[\phi, n]$ is defined by
\begin{gather}
\label{eq:rescaled-borcherds-product-cE}
  \cE\big( B_N[\phi, n] \big)(d)
=
  \sum_{\frakc \in \cC(N),\, N_\frakc(N) = d / n}
  \frac{h_\frakc(N)}{N_\frakc(N)} \cdot \;
  \Big(
   \pi_{\rmF \rmE}\big( tc_0(\phi)_\frakc \big),\,
   \pi_{\rmF \rmE}\big( tc_2(\phi)_\frakc \big)
  \Big)
\text{.}
\end{gather}
The level $N(B[\phi, n])$ of modular forms in the image equals $N$.  Furthermore, we have
\begin{gather}
\label{eq:rescaled-borcherds-product-e}
  e_{q_1}(B_N[\phi, n]) = n \, e_{q_1}(B_N[\phi])
\text{,}
\quad
  e_{\zeta}(B_N[\phi, n]) = n \, e_{\zeta}(B_N[\phi])
\text{, and}
\quad
  e_{q_2}(B_N[\phi, n]) = n \, e_{q_2}(B_N[\phi])
\text{.}
\end{gather}

\section{Siegel Product expansions by Cheng-Duncan}
\label{sec:chengducan}

Consider the following family of weak Jacobi forms $\cZ_g$, the $M_{24}$-twisted elliptic genus.
\begin{gather*}
  \cZ_g(\tau, z)
:=
  \chi(g)\, \frac{\phi_{0,1}(\tau, z)}{12} + \widetilde{T}_g(\tau)\, \phi_{-2, 1}(\tau, z)
\text{,}
\end{gather*}
where $\widetilde{T}_g$ is given in Table~2 of \cite{CD12} and, partially, in Table~\ref{tab:chi-and-tdT}.  We have $tc_0(\cZ_g) = \chi(g)$ and $tc_2(\cZ_g) = \widetilde{T}_g$.  By Proposition~3.1 in~\cite{CD12}, the $\cZ_g$ are weak Jacobi forms of weight~$0$ and index~$1$ for the Jacobi group $\Gamma_0(N_g) \ltimes \ZZ^2$.  That is, $\widetilde{T}_g$ is a modular form of level dividing $N_g$.  We write
\begin{gather}
  \cZ_g(\tau, z)
=
  \sum_{n \ge 0,\, r \in \ZZ} c_g(4 n - r^2) \, q^n \zeta^r
\end{gather}
for the Fourier expansion of $\cZ_g$.

In Section~4 of~\cite{CD12}, Cheng and Duncan give the following product expansion
\begin{gather*}
  \Phi_g(Z)
:=
  q_1 \zeta q_2 \,
  \prod_{(n, r, m) > 0} \exp\Big( -\sum_{k = 1}^\infty \frac{c_{g^k}(4 n m - r^2)}{k} \, (q_1^n \zeta^r q_2^m)^k  \Big)
\text{.}
\end{gather*}
The product runs over triples of integers $(n, r, m)$.  Recall that $(n, r, m) > 0$ means that $n > 0$, or $n = 0$ and $m > 0$, or $n = m = 0$ and $r < 0$.  In analogy with Formula~(3.9) of~\cite{Che10}, we can rewrite the above as
\begin{gather}
\label{eq:chengducan_productexpansion}
  \Phi_g(Z)
=
  q_1 \zeta q_2 \,
  \prod_{d \isdiv n_g}
  \prod_{(n, r, m) > 0} \big( 1 - (q_1^n \zeta^r q_2^m)^d \big)^{c_{g, d}(4 n m - r^2)}
\text{,}
\end{gather}
where
\begin{gather*}
  c_{g, d} (D)
:=
  -d^{-1} \, \sum_{d' \isdiv d} \, \mu(d / d') c_{g^{d'}}(D)
\end{gather*}
and the M\"obius function is denoted by $\mu$.  The $c_{g, d}$ are Fourier coefficients of a weak Jacobi form of weight~$0$ and index~$1$:
\begin{gather}
\label{eq:Zgd}
  \cZ_{g, d}(\tau, z)
=
  \sum_{n \ge 0,\, r \in \ZZ} c_{g, d}(4 n - r^2) \, q^n \zeta^r
:=
  - d^{-1} \, \sum_{d' \isdiv d} \mu(d / d') \cZ_{g^{d'}} (\tau, z)
\text{.}
\end{gather}
Hence $\Phi_g$ is a product expansion to which we can associate $\cE(\Phi_g)$ in the sense of Section~\ref{ssec:siegel-product-expansions}.  It equals
\begin{gather}
\label{eq:cheng-duncan-cE}
  \cE(\Phi_g)(d)
=
  \begin{cases}
    - d^{-1} \, \sum_{d' \isdiv d} \mu(d / d') \cdot \;
    \Big( \chi(g^{d'}), \,
          \widetilde{T}_{g^{d'}} \Big)
   \text{,}
   &
     \text{if $d \isdiv n_g$;}
   \\
    0
   \text{,}
   &
    \text{otherwise.}
  \end{cases}
\end{gather}
Precise expressions for all $g$ that we treat in this paper are given in Table~\ref{tab:Zg}.

\section{Proof of the Main Theorem}
\label{sec:results}

We will prove Theorem~\ref{thm:phig_modularity}, which implies our main result, by applying the following proposition to solutions given in Table~\ref{tab:borcherds_products}.
\begin{proposition}
\label{prop:modularity-criterion}
Suppose that $\Phi$ is a product expansion in the sense of Section~\ref{ssec:siegel-product-expansions} satisfying, for some $0 < N \in \ZZ$ and some $0 < r \in \ZZ$,
\begin{gather*}
  \cE(\Phi)
=
  \sum_{i = 1}^r \cE(B_{N_i}[\phi_i, n_i])
\text{,}
\end{gather*}
where $\phi_i \in \rmJ^{(!)}_{0, 1}(N_i)$ and $n_i N_i \isdiv N$ for all $1 \le i \le r$.  If
\begin{gather*}
  e_{q_1}(\Phi)
=
  \sum_i e_{q_1}( B_{N_i}[\phi_i, n_i])
\text{,}
\quad
  e_{\zeta}(\Phi)
=
  \sum_i e_{\zeta}(B_{N_i}[\phi_i, n_i])
\text{,}
\quad
  e_{q_2}(\Phi)
=
  \sum_i e_{q_2}(B_{N_i}[\phi_i, n_i])
\text{,}
\end{gather*}
then
\begin{gather*}
  \Phi
=
  \prod_{i = 1}^r B_{N_i}[\phi_i, n_i]
\text{.}
\end{gather*}
In particular, $\Phi$ is a Siegel modular forms of level~$N$.
\end{proposition}
\begin{proof}
By Property~\eqref{eq:cE-additivity} of $\cE$, we have
\begin{gather*}
\cE(\Phi)
=
  \cE\big( \prod_{i = 1}^r B_{N_i}[\phi_i, n_i] \big)
\text{.}
\end{gather*}
By definition of $\cE$, we then find
\begin{gather*}
  q_1^{-e_{q_1}(\Phi)}
  \zeta^{-e_{\zeta}(\Phi)}
  q_2^{-e_{q_2}(\Phi)} \;
  \Phi(Z)
=
  q_1^{-\sum_i e_{q_1}(B[\phi_i, n_i])}
  \zeta^{-\sum_i e_{\zeta}(B[\phi_i, n_i])}
  q_2^{-\sum_i e_{q_2}(B[\phi_i, n_i])} \;
  \prod_{i = 1}^r B[\phi_i, n_i](Z)
\text{.}
\end{gather*}
The conditions on $e_{q_1}$, $e_{\zeta}$, and $e_{q_2}$ ensure that we have
\begin{gather*}
  \Phi
=
  \prod_{i = 1}^r B_{N_i}[\phi_i, n_i]
\text{,}
\end{gather*}
as desired.
\end{proof}

\begin{theorem}
\label{thm:phig_modularity}
For
\begin{gather*}
  g
\in
  \big\{ 1\text{A}, 2\text{A}, 2\text{B}, 3\text{A}, 3\text{B},
         4\text{A}, 4\text{B}, 4\text{C}, 5\text{A},
         7\text{AB}, 8\text{A}, 11\text{A}, 23\text{AB}
  \big\}
\text{,}
\end{gather*}
let $p_{3\text{B}} = 3$, $p_{4\text{A}} = 2$, $p_{4\text{C}} = 8$, $p_{8\text{A}} = 8$, and $p_g = 1$ in all other cases.  We have
\begin{gather*}
  \Phi_g^{p_g}
=
  \prod_i B_{N_{g, i}}[p_g\, \phi_{g,i}, n_{g,i}]
\text{,}
\end{gather*}
where $N_{g, i}$, $\phi_{g, i}$, and $n_{g, i}$ are given in Table~\ref{tab:borcherds_products} on page~\pageref{tab:borcherds_products}, and we let
\begin{gather*}
  \phi_{g, i}
=
  tc_0(\phi_{g, i})\, \frac{\phi_{0,1}}{12}
  + tc_2(\phi_{g, i})\, \phi_{-2, 1}
\end{gather*}
in accordance with Isomorphism~\eqref{eq:jacobi-modular-forms-isomorphism}.

In particular, $\Phi_g^{p_g}$ is a Siegel modular form of level~$N_g$ and weight~$p_g k_g$.
\end{theorem}
\begin{corollary}
\label{cor:phig_modularity}
For
\begin{gather*}
  g
\in
  \big\{ 1\text{A}, 2\text{A}, 2\text{B}, 3\text{A}, 3\text{B},
         4\text{A}, 4\text{B}, 4\text{C}, 5\text{A},
         7\text{AB}, 8\text{A}, 11\text{A}, 23\text{AB}
  \big\}
\text{,}
\end{gather*}
the product expansions $\Phi_g$ define meromorphic Siegel modular forms of level~$N_g$ and weight~$k_g$.

In particular, Theorem~\ref{maintheorem} holds.
\end{corollary}
\begin{proof}
Theorem~\ref{thm:phig_modularity} tells us that $\Phi_g^{p_g}$ is modular by representing it as a product of rescaled Borcherds products.  This implies modularity of $\Phi_g$ by the following argument.

Note that the product expansion $\Phi_g(Z)$ is convergent on some domain in $\HS_2$, since its formal power, $\Phi_g^{p_g}$, converges on some domain.  We shall have proved that
\begin{gather*}
  \Phi_g^{p_g}( (A Z + B) (C Z + D)^{-1} )
=
  \det(C Z + D)^{p_g k_g}
  \chi_g(\left(\begin{smallmatrix}A & B \\ C & D \end{smallmatrix}\right))^{p_g}\,
  \Phi_g^{p_g}(Z)
\end{gather*}
for some character $\chi_g$ of $\Gamma_0^{(2)}(N_g)$ and all $\gamma = \left(\begin{smallmatrix}A & B \\ C & D \end{smallmatrix}\right) \in \Gamma_0^{(2)}(N_g)$.  From this we find that
\begin{gather*}
  \Phi_g( (A Z + B) (C Z + D)^{-1} )
=
  \det(C Z + D)^{k_g}
  \epsilon_g(\gamma)
  \chi_g (\gamma)\,
  \Phi_g(Z)
\end{gather*}
for some $\epsilon_g(\gamma)$, which is a character if $g \ne 8\text{A}$, and a multiplier system otherwise.
\end{proof}

\begin{table}[b]
\begin{tabular}{r@{\hspace{3em}}llll}
\toprule
$g$ & $N_{g, i}$ & $n_{g, i}$ & ${\rm tc}_0(\phi_{g, i})$ & ${\rm tc}_2(\phi_{g, i})$ \\
\midrule
$1$\text{A}
    & $1$ & $1$ & $24$ & $0$ \\[3pt]
$2$\text{A}
    & $2$ & $1$ & $8$ & $\rmM(2)_2[ \frac{4}{3} ]$ \\[3pt]
$2$\text{B}
    & $1$ & $1$ & $-12$ & $0$ \\
    & $2$ & $1$ & $12$  & $0$ \\
    & $4$ & $1$ & $0$   & $\rmM_2(4)[ 2, -16 ]$ \\[3pt]
$3$\text{A}
    & $3$ & $1$ & $6$ & $\rmM(3)_2[ \frac{3}{2} ]$ \\[3pt]
$3$\text{B}
    & $1$ & $1$ & $-8$  & $0$ \\
    & $3$ & $1$ & $8$   & $\rmM_2(3)[ 2 ]$ \\[3pt]
4\text{A}
    & $1$ & $1$ & $-6$ & $0$ \\
    & $2$ & $1$ & $2$ & $\rmM(2)_2[ -\frac{2}{3} ]$ \\
    & $4$ & $1$ & $4$ & $\rmM(4)_2[ \frac{2}{3}, 16 ]$ \\
    & $8$ & $1$ & $0$ & $\rmM(8)_2[ 2, 0, -16 ]$ \\[3pt]
4\text{B}
    & $4$ & $1$ & $4$ & $\rmM(4)_2[ \frac{5}{3}, 8 ]$ \\[3pt]
4\text{C}
    & $1$ & $1$ & $-3$ & $0$ \\
    & $2$ & $1$ & $-3$ & $\rmM(2)_2[ \frac{1}{4} ]$ \\
    & $4$ & $1$ & $6$ & $\rmM(4)_2[ \frac{7}{4}, -14 ]$ \\
    & $4$ & $2$ & $0$ & $\rmM(4)_2[ 1, -8 ]$ \\[3pt]
5\text{A}
    & $5$ & $1$ & $4$ & $\rmM(5)_2[ \frac{5}{3} ]$ \\[3pt]
7\text{AB}
    & $7$ & $1$ & $3$ & $\rmM(7)_2[ \frac{7}{4} ]$ \\[3pt]
8\text{A}
    & $1$ & $1$ & $\frac{3}{2}$ & $0$ \\
    & $2$ & $1$ & $-\frac{5}{2}$ & $\rmM(2)_2[ \frac{1}{3} ]$ \\
    & $4$ & $1$ & $1$ & $\rmM(4)_2[ \frac{1}{6}, -12 ]$ \\
    & $8$ & $1$ & $2$ & $\rmM(8)_2[ \frac{4}{3}, 8, 0 ]$ \\
    & $8$ & $2$ & $0$ & $\rmM(8)_2[ 1, 0, -8 ]$ \\[3pt]
11\text{A}
    & $11$ & $1$ & $2$ & $\rmM(11)_2[ \frac{11}{6}, 0 ]$ \\[3pt]
23\text{AB}
    & $23$ & $1$ & $1$ & $\rmM(23)_2[ \frac{23}{12}, \frac{46}{11}, -\frac{23}{11} ]$ \\
\bottomrule
\\
\end{tabular}
\caption{We have $\Phi_g = \prod_i B_{N_{g, i}}[\phi_{g, i}, n_{g, i}]$.  A proof is given in Theorem~\ref{thm:phig_modularity}.}
\label{tab:borcherds_products}
\end{table}
\begin{proof}[Proof of Theorem~\ref{thm:phig_modularity}]
Recall the description of projected cusp expansions given in Section~\ref{ssec:preliminaries:elliptic-modular-forms} and the matrices $\Pi_{\rmF\rmE}(k, N;\, \frakc)$ ($\frakc \in \cC(N)$) given in Table~\ref{tab:projected-cusp-expansions}.  We identify elliptic modular forms ${\rm tc}_2(\phi_{g,i})$ with the coordinate (column) vector with respect to the echelon basis of $\rmM_{2}(\Gamma_0(N_{g, i}))$.  Since ${\rm tc}_0(\phi_{g, i})$ is a constant, its cusp expansions are the same at all $\frakc \in \cC(N_{g, i})$.

To apply Proposition~\ref{prop:modularity-criterion}, we first verify that all $p_g \phi_{g, i}$ satisfy the assumptions of Theorem~\ref{thm:borcherds-products}.  Using \eqref{eq:weak-phi-fourier-expansions} and Table~\ref{tab:projected-cusp-expansions}, it is straightforward to check that the Fourier coefficients $c(p_g \pi_{\rm FE}(\phi_{g, i})_\frakc; \Delta)$ are integral for all $\Delta < 0$.

Next, we compute
\begin{gather}
\label{eq:thm:phig_modularity:Phi_g_sum}
  p_g \cE( \Phi_g )
=
  \sum_i \cE\big( B_{N_{g, i}}[p_g \phi_{g, i}, n_{g, i}] \big)
\text{,}
\end{gather}
using the echelon basis of elliptic modular forms.  By~\eqref{eq:rescaled-borcherds-product-cE}, we have
\begin{gather*}
  \cE\big( B_{N_{g, i}}[p_g \phi_{g, i}, n_{g, i}] \big)(d)
=
  \sum_{\substack{ \frakc \in \cC(N_{g, i}) \\
                   N_\frakc(N_{g, i}) = d \slashdiv n_{g, i} }}
       \frac{h_\frakc(N_{g, i})}{N_\frakc(N_{g, i})}\, p_g
       \Big( {\rm tc}_0(\phi_{g, i}),\;
             \Pi_{\rmF\rmE}(2, N_{g, i}; \frakc) \, {\rm tc}_2(\phi_{g, i}) \Big)
\text{.}
\end{gather*}

Formulas for $e_{q_1}$, $e_{\zeta}$, and $e_{q_2}$ can be derived from Theorem~\ref{thm:borcherds-products} and \eqref{eq:rescaled-borcherds-product-e}.
\begin{gather*}
  e_{q_1}\big( B_{N_{g, i}}[p_g \phi_{g, i}, n_{g, i}] \big)
=
  \frac{n_{g, i}}{24} \sum_{\substack{\frakc \in \cC(N_{g, i}) \\ l \in \ZZ}} \!\!
               h_\frakc(N_{g, i})\, p_g c\big( \pi_{\rm FE}(\phi_\frakc); -l^2 \big)
\end{gather*}
We have $c\big( \pi_{\rm FE}(\phi_\frakc); -l^2 \big) = 0$, if $l \not \in \{0, \pm 1\}$, since $\phi$ is a weak Jacobi forms.  

Therefore,
\begin{alignat*}{2}
&
  e_{q_1}\big( B_{N_{g, i}}[ p_g \phi_{g, i}, n_{g, i}] \big)
\\
={}&
  \frac{n_{g, i}}{24} \sum_{\frakc \in \cC(N_{g, i})} \!\!
               h_\frakc(N_{g, i})
               \Big(
               && \quad p_g {\rm tc}_0(\phi_{g, i})
                  \cdot
                  \big( c( \frac{\phi_{0, 1}}{12};\, 0 ) 
                        + 2 c( \frac{\phi_{0, 1}}{12};\, -1 ) \big)
\\&
               && +
                  c\big( p_g \pi_{\rm FE}( {\rm tc}_2(\phi_{g, i})_\frakc );\, 0 \big)
                  \cdot
                  \big( c( \phi_{-2, 1};\, 0 )
                        + 2\, c( \phi_{-2, 1};\, -1 ) \big)
               \Big)
\text{,}
\end{alignat*}
which equals
\begin{gather*}
  \frac{n_{g, i}}{24} \sum_{\frakc \in \cC(N_{g, i})} \!\!
               h_\frakc(N_{g, i})
               p_g {\rm tc}_0(\phi_{g, i})
\text{.}
\end{gather*}
A similar computations gives
\begin{align*}
&  e_{\zeta}\big( B_{N_{g, i}}[p_g \phi_{g, i}, n_{g, i}] \big)
=
  e_{q_2}\big( B_{N_{g, i}}[p_g \phi_{g, i}, n_{g, i}] \big)
\\
={}&
  \frac{n_{g, i}}{2} \sum_{\frakc \in \cC(N_{g, i})} \!\!
              h_\frakc(N_{g, i})
              \big(
              \frac{1}{12} p_g {\rm tc}_0(\phi_{g, i})
              + 
              p_g c\big( \Pi_{\rm FE}(2, N_{g, i}; \frakc) {\rm tc}_2(\phi_{g, i}) ;\, 0 \big)
              \big)
\text{.}
\end{align*}
The $0$th coefficient of a modular form can be easily read of from its coordinates with respect to the echelon basis.  For this reasons, the above expressions can be directly evaluated using Table~\ref{tab:projected-cusp-expansions}.

The weight $p_g k_g$ of $\prod_i B_{N_{g, i}}[p_g \phi_{g, i}, n_{g, i}]$ can be computed along the same line.
By Theorem~\ref{thm:borcherds-products}, we have
\begin{gather*}
  p_g k_g
=
  \frac{1}{2} \sum_i \sum_{\frakc \in \cC(N_{g, i})}
                      \frac{h_\frakc(N_{g, i})}{N_\frakc(N_{g, i})}
                      \Big( p_g {\rm tc}_0(p_g\, \phi_{g, i}) 
                          + p_g c\big( \Pi_{\rmF\rmE}(2, N_{g, i}; \frakc) {\rm tc}_2(p_g\, \phi_{g, i}); 0\big) \Big)
\text{.}
\end{gather*}

We have reduced all computations to straight forward linear algebra.  We give details only in the case of $g = 3\text{B}$, and leave all other cases to the reader.  For $g = 3\text{B}$, we have
\begin{align*}
  \phi_{3\text{B}, 1}(\tau, z)
&=
  -8 \frac{\phi_{0, 1}(\tau, z)}{12}
&\text{with}\quad
  N_{3\text{B}, 1} = 1
\text{,}\,
  n_{3\text{B}, 1} = 1
\text{;}
\\
  \phi_{3\text{B}, 2}(\tau, z)
&=
  8 \frac{\phi_{0, 1}(\tau, z)}{12}
  + \rmM_2(3)[2](\tau) \phi_{-2, 1}(\tau, z)
&\text{with}\quad
  N_{3\text{B}, 1} = 3
\text{,}\,
  n_{3\text{B}, 1} = 1
\text{.}
\end{align*}
Since $N_{3\text{B}, i} \isdiv 3$ for $i \in \{1, 2\}$, we only need to consider forms of level $1$ and $3$.  The cusp data for $\Gamma_0(3)$ is:
\begin{gather*}
  h_\infty(3) = 1
\text{,}\;
  N_\infty(3) = 1
\text{;}\quad
  h_0(3) = 4
\text{,}\;
  N_0(3) = 3
\text{.}
\end{gather*}
By definition, $\Pi_{\rm FE}(2, N; \infty)$ is the identity for all $N$.  In addition, we need $\Pi_{\rm FE}(2, 3; 0) = (- \frac{1}{3})$.  We consider the third power of $\Phi_g$ corresponding to $p_{3\text{B}} = 3$.  Putting everything together, we obtain
\begin{align*}
  \cE\big( B_1[p_{3\text{B}}\, \phi_{3\text{B}, 1}, 1] \big) (1)
&=
  \frac{h_\infty(1)}{N_\infty(1)}
  \big( 3 \cdot (-8),\, 0 \big)
=
  (-24, 0)
\text{,}
\\
  \cE\big( B_3[p_{3\text{B}}\, \phi_{3\text{B}, 2}, 1] \big) (1)
&=
  \frac{h_\infty(3)}{N_\infty(3)}
  \big( 3 \cdot 8,\, 3 \cdot \rmM_2(3)[2] \big)
=
  \big( 24,\, \rmM_2(3)[6] \big)
\text{,}
\\
  \cE\big( B_3[p_{3\text{B}}\, \phi_{3\text{B}, 2}, 1] \big) (3)
&=
  \frac{h_\infty(3)}{N_\infty(3)}
  \big( 3 \cdot 8,\, 3 \cdot (-\tfrac{1}{3})\, \rmM_2(3)[2] \big)
=
  \big( 24, \rmM_2(3)[-2] \big)
\text{.}
\end{align*}

Comparing this with
\begin{gather*}
  p_{3\text{B}}\, \cE(\Phi_{3\text{B}})(1) = (0,\, \rmM_2(3)[6])
\text{,}\quad
  p_{3\text{B}}\, \cE(\Phi_{3\text{B}})(3) = (24,\, \rmM_2(3)[-2])
\text{,}
\end{gather*}
we establish
\begin{gather*}
  p_{3\text{B}}\, \cE(\Phi_{3\text{B}})
=
  \cE\big( B_1[p_{3\text{B}}\, \phi_{3\text{B}, 1}, 1] \big)
  + \cE\big( B_3[p_{3\text{B}}\, \phi_{3\text{B}, 2}, 1] \big)
\text{.}
\end{gather*}

Next, we find that 
\begin{align*}
  e_{q_1}\big( B_1[p_{3\text{B}}\, \phi_{3\text{B}, 1}, 1] \big)
&=
  \frac{1}{24} h_\infty(1) \cdot 3 \cdot (-8)
=
  -1
\text{,}
\\
  e_{q_1}\big(B_3[p_{3\text{B}}\, \phi_{3\text{B}, 2}, 1] \big)
&=
  \frac{1}{24} (h_\infty(3) + h_0(3)) \cdot 3 \cdot 8
=
  4
\text{,}
\end{align*}
and
\begin{align*}
  e_{\zeta}\big( B_1[p_{3\text{B}}\, \phi_{3\text{B}, 1}, 1] \big)
&=
  \frac{1}{2} h_\infty(1) \cdot 3 \cdot (-\frac{8}{12})
=
  -1
\text{,}
\\
  e_{\zeta}\big( B_3[p_{3\text{B}}\, \phi_{3\text{B}, 2}, 1] \big)
&=
  \frac{1}{2} \big( h_\infty(3) \big( 3 \cdot \frac{8}{12} + c\big( 3 \rmM_2(3)[2]; 0 \big) \big)
                    +
                    h_0(3) \big( 3 \cdot \frac{8}{12} + c\big( \Pi_{\rm FE}(2, 3; 0)\, 3 \rmM_2(3)[2];\, 0 \big) \big)
\\
&=
  \frac{1}{2} \big( ( 3 \cdot \frac{8}{12} + 6 )
                    +
                    3 ( 3 \cdot \frac{8}{12} + (-\frac{1}{3}) 6 ) \big)
=
  4
\text{,}
\end{align*}
while
\begin{gather*}
  3 e_{q_1}(\Phi_{3\text{B}})
=
  3 e_{\zeta}(\Phi_{3\text{B}})
=
  3 e_{q_2}(\Phi_{3\text{B}})
=
  3
\text{.}
\end{gather*}
This proves that $\Phi_{3\text{B}}^{p_{3\text{B}}}$ is modular.
\end{proof}

\section{Conclusion}
\label{sec:conclusion}

We have shown that $M_{24}$-twisted product expansions that arise form second quantized elliptic genera of $\text{K}3$ are modular, if the twisting element has non-composite order.  As opposed to results that were obtained so far, this modularity does not directly come from Borcherds products.  Instead, we have illustrated how products of rescaled Borcherds products enter the picture.  This is a new type of modularity that has not yet been observed in string theory.  It is, however, a natural generalization of the concept of eta products to Siegel modular forms.

The analog of the Weyl denominator formula for Kac-Moody algebras remains to be found.  Since its multiplicative part is more complicated than being just one Borcherds products, we speculate that the additive side might also be more complicated.  It is possible that it is the sum of rescaled additive lifts.  We have computed the weight of the twisted product expansions $\Phi_g$ in all cases that we could resolve.  In the cases $g = 11\text{A}$ and $g = 23\text{AB}$ the inequality $k_g \le 0$ implies that $\Phi_g$ is meromorphic.  Therefore, it will be necessary to take the regularized additive lift~\cite{Bo98} into consideration, while searching for the additive side of the attached Weyl denominator formula.  A physical interpretation for the divisors that occur would be of certain interest.


Finally, the composite levels~$N_g$ are of seemingly different nature.  The easiest case is $g = 6\text{A}$ in which case the predicted level is $N_{6\text{A}} = 6$.  We have tried -- without success -- to represent $\Phi_{6\text{A}}$ as a product of rescaled Borcherds products coming from weakly holomorphic Jacobi forms with pole order less than~$4$.  Clearly, novel ideas are needed to make progress in this direction.

\begin{appendix}

\section{How we found the data in Table~\ref{tab:borcherds_products}}
\label{appendix:how-we-found-the-data}

In order to express some power of $\Phi_g$ as products of rescaled Borcherds products, we solved the system of linear equations associated to the ansatz
\begin{gather}
\label{eq:ansatz-for-phi_g}
  \cE(\Phi_g)
=
  \sum_{\substack{n N \isdiv N'_g \\ N \le N_{\max}}} \cE(B_N(\phi_{N, n}, n))
\text{,}
\end{gather}
where $\phi_{N,n} \in \rmJ^!_{k, m}(\Gamma_0(N))$ with maximal pole order $0 \le o \in \ZZ$ at all cusps.  This is a finite dimensional system of linear equations, since $\Delta^o \phi_{N, n} \in \rmJ_{12 o, 1}^{(!)}(\Gamma_0(N))$, and $\dim\, \rmJ_{12 o, 1}^{(!)}(\Gamma_0(N)) < \infty$.  Note that for all cases that we have successfully solved, we have $N'_g = N_g$ and $o = 0$, while $N_{\max}$ equals the level of $\widetilde{T}_g$.

We used Sage~\cite{sage} (mind the version of Sage, which might be relevant to build and run our implementation) to build the matrix attached to~\eqref{eq:ansatz-for-phi_g} and solve it.  The code that we have used can be downloaded at~\cite{Rhomepage}.  We briefly describe how to invoke relevant functions in order to reprove our results.  At this point, we mention that, at~\cite{Rhomepage}, also the results of can be downloaded in machine readable form.

As a first step, download the code for this paper at~\cite{Rhomepage}.  Extract ``cheng\_products.tar.gz'', and run the following commands in terminal:
\begin{lstlisting}[basicstyle=\small\ttfamily]
cd eta_products
sage -sh
sage --python compile.py
exit
\end{lstlisting}

The basic solving procedure is invoked by {\tt cheng\_solve\_case}, whose declaration looks as follows.
\begin{lstlisting}[basicstyle=\small\ttfamily]
cheng_solve_case( g, conjectured_Ng = None, max_phi_level = None,
                  order = 0, minimize_scaling = False )
\end{lstlisting}
The first argument is the label of a conjugacy class that one wants to let compute.  The second argument, which is optional, is the conjectured level~$N'_g$ of $\Phi_g$.  If this parameter is not specified it will be replaced by the levels~$N_g$ conjectured in~\cite{CD12}.  The third parameter corresponds to $N_{\max}$.  Clearly, we must have $N_{\max} \isdiv N'_g$.  The default value in this case is the level of $\widetilde{T}_g$.  The next argument corresponds to $o$ in~\eqref{eq:ansatz-for-phi_g}.  The default value equals~$0$.  The last parameter determines whether or not the returned solution should involve as few rescaled Borcherds products as possible.

The next lines of code, compute the solution in the case $g = 11\text{A}$.
\begin{lstlisting}[basicstyle=\small\ttfamily]
sage: %runfile cheng_products.py
sage: cheng_solve_case('11A', minimize_scaling = True)
('11A',
 11,
 11,
 (0, 2, 11/6, 0, 0),
 [(1, 1, 0), (11, 1, 0), (11, 1, 1), (11, 1, 2), (1, 11, 0)])
sage: compute_borcherds_lift_data(_)                                   
(0, (1, 1, 1), {(11, 1, 0): (2, 0), (11, 1, Infinity): (-2, 2)})
\end{lstlisting}
The return value of the first command is $(g, N'_g, N_{\max}, v, l)$, where $v$ is a vector whose coordinates correspond to the elements of the list $l$ of triples $(N, n, i)$.  The first and second component of this triple are the same as in~\eqref{eq:ansatz-for-phi_g}.  The third component refers to the $(i + 1)$th basis vector in the echelon basis of $\rmM^{!(o)}_0(N)$, if $i$ is less than its dimension, and to the $(i + 1 - \dim\, \rmM^{!(o)}_0(N))$th echelon basis vector of $\rmM^{!(o)}_{2}(N)$, otherwise.  Here we write $\rmM^{!(o)}_k(N)$ for weakly holomorphic modular forms of weight~$k$ and level~$N$ whose pole order at all cusps is bounded by~$o$.  In the given example only $N = 11,\, n = 1$ contributes to the solution, and we have
\begin{gather*}
  \phi_{11, 1}
\cong
  \big(\rmM^{!(0)}_0(N)[2],\, \rmM^{!(0)}_2(N)[\tfrac{11}{6}]\big)
\end{gather*}
via Isomorphism~\eqref{eq:jacobi-modular-forms-isomorphism}.

To ensure that~\eqref{eq:ansatz-for-phi_g} corresponds to an equality of
\begin{gather*}
  \Phi_g^p
=
  \prod_{\substack{n N \isdiv N'_g \\ N \le N_{\max}}} \cE(B_N(p \phi_{N, n}, n))
\text{,}
\end{gather*}
the second command computes the weight $k$ of the right hand side, $(e_{q_1}, e_\zeta, e_{q_2})$, and a dictionary mapping $(N, n, \frakc)$ to the pair of  Fourier coefficients
\begin{gather*}
  \big( c((\phi_{N, n})_\frakc; 0),\, c((\phi_{N, n})_\frakc; -1) \big)
\text{.}
\end{gather*}
Note that this function is currently restricted to the case of $o = 0$, but it would be easy to implement a general form.

\section{Projection matrices $\Pi_{\rm FE}$}
\label{appendix:projection-matrices}

The matrices $\Pi_{\rm FE}(k, N, \frakc)$ given in Table~\ref{tab:projected-cusp-expansions} have been computed using methods in~\cite{Ra13}, which are algebraic.  We provide numerical double checks, to verify their correctness.  Follow the initial steps described in the previous section.  Install ``nose`` for Sage by typing
\begin{lstlisting}[basicstyle=\small\ttfamily]
sage -i nose
\end{lstlisting}
Then run the corresponding tests via
\begin{lstlisting}[basicstyle=\small\ttfamily]
sage -sh
nosetests -v modular_form_transformations__test.py
\end{lstlisting}
This will automatically perform all tests, which the reader can inspect by viewing the file ``modular\_form\_transformations\_test.py``.

\section{Tables}
\label{appendix:tables}

\begin{table}[H]
\begin{tabular}{l@{\hspace{2em}}llll}
\toprule
$g$ & $\chi(g)$ & ${\td T}_g$ \\
\midrule
$1\text{A}$ & $24$ & $0$ \\[2pt]
$2\text{A}$ & $8$  & $\rmM_2(2)[\frac{4}{3}]$ \\[2pt]
$2\text{B}$ & $0$  & $\rmM_2(4)[2, -16]$ \\[2pt]
$3\text{A}$ & $6$ & $\rmM_2(3)[\frac{3}{2}]$ \\[2pt]
$3\text{B}$ & $0$ & $\rmM_2(3)[2]$ \\[2pt]
$4\text{A}$ & $0$ & $\rmM_2(8)[2, 0, -16]$ \\[2pt]
$4\text{B}$ & $4$ & $\rmM_2(4)[\frac{5}{3}, 8]$ \\[2pt]
$4\text{C}$ & $0$ & $\rmM_2(4)[2, -8]$ \\[2pt]
$5\text{A}$ & $4$ & $\rmM_2(5)[\frac{5}{3}]$ \\[2pt]
$7\text{AB}$ & $3$ & $\rmM_2(7)[\frac{7}{4}]$ \\[2pt]
$8\text{A}$ & $2$ & $\rmM_2(8)[\frac{11}{6}, 4, 12]$ \\[2pt]
$11\text{A}$ & $2$ & $\rmM_2(11)[\frac{11}{6}, 0]$ \\[2pt]
$23\text{AB}$ & $1$ & $\rmM_2(23)[\frac{23}{12}, \frac{46}{11}, -\frac{23}{11}]$ \\
\bottomrule
\\
\end{tabular}
\caption{Cheng's and Duncan's initial data $\chi(g)$ and ${\td T}_g$, expressed in terms of the echelon bases.}
\label{tab:chi-and-tdT}
\end{table}

\begin{table}[h]
\begin{tabular}{l@{\hspace{2em}}lll@{\hspace{2em}}lll}
\toprule
$N$ & $\frakc$ & $h_\frakc$ & $N_\frakc$ & $\frakc$ & $h_\frakc$ & $N_\frakc$ \\
\midrule
$2$ &   $\infty$ & $1$ & $1$ &   $0$ & $2$ & $2$ \\[3pt]
$3$ &   $\infty$ & $1$ & $1$ &   $0$ & $3$ & $3$ \\[3pt]
$4$ &   $\infty$ & $1$ & $1$ &   $0$ & $4$ & $4$ \\
    &   $\frac{1}{2}$ & $1$ & $2$ \\[3pt]
$5$ &   $\infty$ & $1$ & $1$ &   $0$ & $5$ & $5$ \\[3pt]
$7$ &   $\infty$ & $1$ & $1$ &   $0$ & $7$ & $7$ \\[3pt]
$8$ &   $\infty$ & $1$ & $1$ &   $0$ & $8$ & $8$ \\
    &   $\frac{1}{2}$ & $2$ & $4$ &   $\frac{1}{4}$ & $1$ & $2$ \\[3pt]
$11$ &   $\infty$ & $1$ & $1$ &   $0$ & $11$ & $11$ \\[3pt]
$23$ &   $\infty$ & $1$ & $1$ &   $0$ & $23$ & $23$ \\
\bottomrule
\\
\end{tabular}
\caption{Cusp data for $\Gamma_0(N)$, used in Section~\ref{sec:extendedborcherdsproducts}.}
\label{tab:cups-data}
\end{table}
\clearpage

\begin{table}[h]
\begin{tabular}{l@{\hspace{2em}}llll}
\toprule
$g$ & $d$ & $\cE(\Phi_g)(d)$ & $d$ & $\cE(\Phi_g)(d)$ \\
\midrule
$1\text{A}$ &  $1$ & $\big( 24,\, 0 \big)$ \\[3pt]
$2\text{A}$ &  $1$ & $\big( 8,\, \rmM_2(2)[\frac{4}{3}] \big)$ &
               $2$ & $\big( 8,\, \rmM_2(2)[-\frac{2}{3}] \big)$ \\[3pt]
$2\text{B}$ &  $1$ & $\big( 0,\, \rmM_2(4)[2, -16] \big)$ &
               $2$ & $\big( 12,\, \rmM_2(4)[-1, 8] \big)$ \\[3pt]
$3\text{A}$ &  $1$ & $\big( 6,\, \rmM_2(3)[\frac{3}{2}] \big)$ &
               $3$ & $\big( 6,\, \rmM_2(3)[-\frac{1}{2}] \big)$ \\[3pt]
$3\text{B}$ &  $1$ & $\big( 0,\, \rmM_2(3)[2] \big)$ &
               $3$ & $\big( 8,\, \rmM_2(3)[-\frac{2}{3}] \big)$ \\[3pt]
$4\text{A}$ &  $1$ & $\big( 0,\, \rmM_2(8)[2, 0, -16] \big)$
            &  $2$ & $\big( 4,\, \rmM_2(8)[-\frac{1}{3}, 16, 24] \big)$ \\
            &  $4$ & $\big( 4,\, \rmM_2(8)[-\frac{1}{3}, -8, -8] \big)$ \\[3pt]
$4\text{B}$ &  $1$ & $\big( 4,\, \rmM_2(4)[\frac{5}{3}, 8] \big)$
            &  $2$ & $\big( 2,\, \rmM_2(4)[-\frac{1}{6}, 12] \big)$ \\
            &  $4$ & $\big( 4,\, \rmM_2(4)[-\frac{1}{3}, -8] \big)$ \\[3pt]
$4\text{C}$ &  $1$ & $\big( 0,\, \rmM_2(4)[2, -8] \big)$
            &  $2$ & $\big( 0,\, \rmM_2(4)[0, -4] \big)$ \\
            &  $4$ & $\big( 6,\, \rmM_2(4)[-\frac{1}{2}, 4] \big)$ \\[3pt]
$5\text{A}$ &  $1$ & $\big( 4,\, \rmM_2(5)[\frac{5}{3}] \big)$
            &  $5$ & $\big( 4,\, \rmM_2(5)[-\frac{1}{3}] \big)$ \\[3pt]
$7\text{AB}$ &  $1$ & $\big( 3,\, \rmM_2(7)[\frac{7}{4}] \big)$
            &  $7$ & $\big( 3,\, \rmM_2(7)[-\frac{1}{4}] \big)$ \\[3pt]
$8\text{A}$ &  $1$ & $\big( 2,\, \rmM_2(8)[\frac{11}{6}, 4, 12] \big)$
            &  $2$ & $\big( -1,\, \rmM_2(8)[\frac{1}{12}, -2, -14] \big)$ \\
            &  $4$ & $\big( 2,\, \rmM_2(8)[-\frac{1}{6}, 8, 12] \big)$
            &  $8$ & $\big( 2,\, \rmM_2(8)[-\frac{1}{6}, -4, -4] \big)$ \\[3pt]
$11\text{A}$ &  $1$ & $\big( 2,\, \rmM_2(11)[\frac{11}{6}, 0] \big)$
            &  $11$ & $\big( 2,\, \rmM_2(11)[-\frac{1}{6}, 0] \big)$ \\[3pt]
$23\text{AB}$ &  $1$ & $\big( 1,\, \rmM_2(23)[\frac{23}{12}, \frac{46}{11}, -\frac{23}{11}] \big)$
            &  $23$ & $\big( 1,\, \rmM_2(23)[-\frac{1}{12}, -\frac{2}{11}, \frac{1}{11}] \big)$ \\
\bottomrule
\\
\end{tabular}
\caption{Values of $\cE(\Phi_g)$ which are used in the course of the proof of Theorem~\ref{thm:phig_modularity}.}
\label{tab:Zg}
\end{table}

\begin{table}[h]
\begin{tabular}{ll@{\hspace{2em}}ll@{\hspace{2em}}ll}
\toprule
$k$ & $N$ & $\frakc$ & $\Pi_{\rm FE}(k, N, \frakc)$ & $\frakc$ & $\Pi_{\rm FE}(k, N, \frakc)$ \\
\midrule
$2$ & $2$ & $0$ & $\begin{pmatrix} -\frac{1}{2} \end{pmatrix}$ \\[4pt]
$2$ & $3$ & $0$ & $\begin{pmatrix} -\frac{1}{3} \end{pmatrix}$ \\[4pt]
$2$ & $4$ & $0$ & $\begin{pmatrix} -\frac{1}{8} & -\frac{1}{64} \\
                                  -3 & -\frac{3}{8} \end{pmatrix}$
          & $\frac{1}{2}$ & $\begin{pmatrix} -\frac{1}{2} & \frac{1}{16} \\
                                             12 & \frac{1}{2} \end{pmatrix}$ \\[4pt]
$2$ & $5$ & $0$ & $\begin{pmatrix} -\frac{1}{5} \end{pmatrix}$ \\[4pt]
$2$ & $7$ & $0$ & $\begin{pmatrix} -\frac{1}{7} \end{pmatrix}$ \\[4pt]
$2$ & $8$ & $0$ & $\begin{pmatrix} -\frac{1}{32} & -\frac{1}{64} & -\frac{1}{256} \\
                                   -\frac{3}{4} & -\frac{3}{8} & -\frac{3}{32} \\
                                   -\frac{3}{4} & -\frac{3}{8} & -\frac{3}{32} \end{pmatrix}$
          & $\frac{1}{2}$ & $\begin{pmatrix} -\frac{1}{8} & \frac{1}{16} & -\frac{1}{64} \\
                                             3 & \frac{1}{2} & \frac{3}{8} \\
                                             -3 & \frac{3}{2} & -\frac{3}{8} \end{pmatrix}$ \\[15pt]
&         & $\frac{1}{4}$ & $\begin{pmatrix} -\frac{1}{2} & 0 & \frac{1}{16} \\
                                             0 & 1 & 0 \\
                                             12 & 0 & \frac{1}{2} \end{pmatrix}$ \\
$2$ & $11$ & $0$ & $\begin{pmatrix} -\frac{1}{11} & 0 \\
                                    0 & -\frac{1}{11} \end{pmatrix}$ \\[4pt]
$2$ & $23$ & $0$ & $\begin{pmatrix} -\frac{1}{23} & 0 & 0 \\
                                    0 & -\frac{1}{23} & 0 \\
                                    0 & 0 & -\frac{1}{23} \end{pmatrix}$ \\[4pt]
\bottomrule
\end{tabular}
\caption{Matrices $\Pi_{\rm FE}(k, N, \frakc)$ associated to $\pi_{\rm FE}(\,\cdot\,_{\frakc}) : \rmM_k(N) \rightarrow \rmM_k(N)$; see Section~\ref{ssec:preliminaries:elliptic-modular-forms}.}
\label{tab:projected-cusp-expansions}
\end{table}

\end{appendix}

\bibliographystyle{amsalpha}
\bibliography{bibliography}

\end{document}

%% file: shortcuts.tex
\newcommand{\tit}{\itshape}

\renewcommand{\frak}{\ensuremath{\mathfrak}}
\newcommand{\cal}{\ensuremath{\mathcal}}

\newcommand{\frakc}{\ensuremath{\frak{c}}}

\newcommand{\frake}{\ensuremath{\frak{e}}}

\newcommand{\cC}{\ensuremath{\cal{C}}}

\newcommand{\cE}{\ensuremath{\cal{E}}}

\newcommand{\cZ}{\ensuremath{\cal{Z}}}

\newcommand{\rmE}{\ensuremath{\mathrm{E}}}
\newcommand{\rmF}{\ensuremath{\mathrm{F}}}

\newcommand{\rmJ}{\ensuremath{\mathrm{J}}}

\newcommand{\rmM}{\ensuremath{\mathrm{M}}}

\newcommand{\amid}{\ensuremath{\mathop{\mid}}}

\newcommand{\ZZ}{\ensuremath{\mathbb{Z}}}
\newcommand{\QQ}{\ensuremath{\mathbb{Q}}}
\newcommand{\RR}{\ensuremath{\mathbb{R}}}
\newcommand{\CC}{\ensuremath{\mathbb{C}}}

\renewcommand{\Im}{\ensuremath{\mathop{\mathfrak{Im}}}}

\newcommand{\isdiv}{\amid}

\newcommand{\Mat}[2]{\ensuremath{\mathrm{Mat}_{#1}(#2)}}

\newcommand{\SL}[1]{\ensuremath{\mathrm{SL}_{#1}}}
\newcommand{\Sp}[1]{\ensuremath{\mathrm{Sp}_{#1}}}

\newcommand{\T}{\ensuremath{\mathrm{T}}}

\newcommand{\slashdiv}{\ensuremath{\mathop{/}}}

\newcommand{\HS}{\mathbb{H}}

\newcommand{\td}{\tilde}

\renewcommand{\gcd}{\ensuremath{\mathrm{gcd}}}

\renewcommand{\pmod}[1]{\ensuremath{\;({\rm mod}\;#1)}}